\documentclass{amsart}

\usepackage{amscd,amsmath,amssymb,amsthm,mathrsfs,latexsym}
\usepackage[mathcal]{euscript}
\usepackage{mathrsfs}

\usepackage[all]{xy}
\usepackage{color,soul} 
\usepackage{hyperref} 
\hypersetup{
 colorlinks=true, 	 
 linktoc=all,     	 
 linkcolor=blue,  	 
 citecolor=blue,    
 filecolor=magenta,  
 urlcolor=black       
}

\usepackage{tikz-cd}

\usepackage{enumerate}

\newtheorem{thm}{Theorem}[section]
\newtheorem{cor}[thm]{Corollary}
\newtheorem{lem}[thm]{Lemma}
\newtheorem{prop}[thm]{Proposition}

\theoremstyle{definition}

\newtheorem{rmk}[thm]{Remark}

\theoremstyle{definition}
\newtheorem{ex}[thm]{Example}
\theoremstyle{remark}

\newtheoremstyle{named}{}{}{\itshape}{}{\bfseries}{.}{.5em}{\thmnote{#3's }#1}
\theoremstyle{named}

\newcommand{\R}{{\mathbb R}}
\newcommand{\C}{{\mathbb C}}
\newcommand{\Z}{{\mathbb Z}}

\newcommand{\T}{{\mathcal T}}
\newcommand{\Sg}{{{\mathcal S}}}
\newcommand{\X}{{\rm X}}
\newcommand{\Rep}{{\rm Rep}}
\newcommand{\SL}{{\rm SL}}
\newcommand{\GL}{{\rm GL}}
\newcommand{\U}{{\rm U}}
\newcommand{\SU}{{\rm SU}}
\newcommand{\SO}{{\rm SO}}

\newcommand{\Q}{{\mathbb Q}}
\newcommand{\F}{{\mathbb F}}

\newcommand{\B}{{\mathcal B}}

\newcommand{\rk}{{\rm rank\,}}

\newcommand{\Hom}{{\rm Hom}}
\newcommand{\Comm}{{\rm Comm}}

\newcommand{\leqs}{\leqslant}

\newcommand{\maps}{\longrightarrow}

\newenvironment{psmallmatrix}
  {\left(\begin{smallmatrix}}
  {\end{smallmatrix}\right)}

\usepackage{mathtools} 

\def\ds{\displaystyle}

\title{Poincar\'e series of character varieties for nilpotent groups}

\author{Mentor Stafa}
\address{Mathematical Sciences, Indiana University - Purdue University Indianapolis, IN 46202}
\email{mstafa@iupui.edu}

\date{\today}

\begin{document}

\begin{abstract}
For any compact and connected Lie group $G$ and any free abelian or 
free nilpotent group $\Gamma$ , we determine the cohomology of the 
path component of the trivial representation of the representation space 
(character variety) $\Rep(\Gamma,G)_1$, with coefficients in a 
field $\F$ with $\rm char (\F)$ either 0 or relatively prime to 
the order of the Weyl group $W$. We give explicit 
formulas for the Poincar\'e series. In addition we study 
$G$-equivariant stable decompositions of subspaces $\X(q,G)$ 
of the free monoid $J(G)$ generated by the Lie group $G$, obtained from
finitely generated free nilpotent group representations. 
\end{abstract}

\maketitle

\tableofcontents

\section{Introduction}

Let $G$ be a compact and connected Lie group 
and $\Gamma$ be a finitely generated discrete group.
In this article we study 
representation spaces given by orbit spaces
$\Rep(\Gamma,G)=\Hom(\Gamma,G)/G$ 
obtained from spaces of homomorphisms $\Hom(\Gamma,G)\subseteq G^n$ 
endowed with the subspace topology of $G^n$, modulo the conjugation action of $G$.
If $G$ is compact the representation space $\Rep(\Gamma,G)$ coincides with
the character variety $\Hom(\Gamma,G)/\!\!/G$,
sometimes also denoted by $\mathfrak{X}_G(\Gamma)$. 
On the other hand, if $G$ is the group of complex or real points of a 
reductive linear algebraic group and $C \leqs G$ is a
maximal compact subgroup, Bergeron~\cite[Theorem II]{bergeron}
showed that for $\Gamma$ finitely generated nilpotent, there is 
a homotopy equivalence $\Hom(\Gamma,G)/\!\!/G \simeq \Hom(\Gamma,C)/C$ (for non-examples see~\cite{adem2007commuting}).
Of particular interest here are the cases when $\Gamma$ is
either a free abelian or free nilpotent group.
For these groups we give a complete answer for the Poincar\'e series
of the connected components of the trivial representation
$\Rep(\Gamma,G)_1\subseteq \Rep(\Gamma,G)$.
The answer depends only on the Weyl group $W$ of $G$ and its
action on the maximal torus $T$.
Ramras and Stafa~\cite{ramrasstafa} give an analogous
complete answer for the Poincar\'e series
of the connected components of the trivial representation
$\Hom(\Gamma,G)_1\subseteq \Hom(\Gamma,G)$.
The answer for $\Hom(\Gamma,G)_1$ is governed by the maximal
torus $T$ of rank $r$, the Weyl group $W$ and its corresponding characteristic
degrees $d_1,\dots,d_r$ as follows
$$
P(\Hom(\Z^n,G)_1;q)=|W|^{-1} \prod_{i=1}^r (1-q^{2d_i}) 
\sum_{w\in W} \frac{\det(1+qw)^n}{\det(1-q^2w)} .
$$
In our case characteristic degrees do not appear in the 
Poincar\'e series of $\Rep(\Gamma,G)_1$.

The topology of the spaces $\Hom(\Gamma,G)$ and $\Rep(\Gamma,G)$ for 
free abelian and nilpotent groups (and certainly other discrete groups) 
has attracted considerable attention recently \cite{stafa.comm.2}.
The space of homomorphisms $\Hom(\Z^n, G)$ is known as
{the space of ordered pairwise commuting $n$-tuples in $G$}, and the representation
space $\Rep(\Z^n,G)$ can be identified with the moduli space of
{isomorphism classes of flat connections on principal $G$-bundles over
the $n$-torus}. These spaces and their variations including the space of
almost commuting elements \cite{borel2002almost},
have also been studied in various settings outside topology, most notably including
work of  Witten \cite{witten1,witten2} and Kac--Smilga \cite{kac.smilga}  
on supersymmetric Yang-Mills theory.

\subsection*{Statements of main results}

Let $F_n$ be the free group on $n$ letters, with 
{descending central series} given by
$\cdots \unlhd \Gamma^{3} \unlhd \Gamma^2 \unlhd  \Gamma^1 = F_n.$
Then there is a filtration of $G^n$ by spaces of homomorphisms
$$
\Hom(F_n/\Gamma^2,G) \subseteq \Hom(F_n/\Gamma^3,G) \subseteq \Hom(F_n/\Gamma^4,G) \subseteq \cdots \subseteq G^n.
$$
The main variation of the spaces $\Hom(F_n/\Gamma^q,G)$ studied in this paper are spaces
of almost commuting or almost nilpotent $n$-tuples in $G$,
which are defined in Section~\ref{sec: Definitions}. 
Let $J(G)$ be the free monoid generated by $G$, also known as
the James reduced product on $G$. The descending central series above
can be used to define a sequence of spaces $\X(q,G)$ that filter $J(G)$
$$
\X(2,G) \subseteq \X(3,G) \subseteq \X(4,G) \subseteq \cdots \subseteq J(G)
$$
also defined in Section~\ref{sec: Definitions}, previously defined and used in work of 
Cohen--Stafa~\cite{stafa.comm} and Ramras--Stafa~\cite{ramrasstafa}.
In particular, we denote the space $\X(2,G)$ by $\Comm(G)$.
$G$ acts by conjugation on the elements in $\Hom(F_n/\Gamma^q,G)$ 
and elements (or words) in $\X(q,G)$.
We show that after one suspension the orbit spaces $\X(q,G)/G$
decompose into infinite wedge sums of \textit{smaller} spaces.

\begin{thm}\label{thm: decomposition of X(q,G) INTRO }
Let $G$ be a compact and connected Lie group. 
For each $q\geqslant  2$ there is a homotopy equivalence
$$\Sigma ( \X(q,G)/G) \simeq \Sigma 
		\bigvee_{n \geqslant  1} \widehat{\Hom}(F_n/\Gamma^q,G)/G. $$
\end{thm}

Here the space $\widehat{\Hom}(F_n/\Gamma^q,G)$ is the quotient
of the space of homomorphisms ${\Hom}(F_n/\Gamma^q,G)$ 
by the subspace consisting of $n$-tuples
with at least one coordinate the identity. Note that these decompositions 
hold also for the component of the trivial representation, which we denote
by $\X(q,G)_1$ and $\Comm(G)_1$, as observed for instance in
\cite{ramrasstafa}. This is explained in more detail below.
The main theorem that makes it possible for the calculation 
of the Poincar\'e series is the following.

\begin{thm}\label{thm: Comm(G)/G = J(T)/W INTRO}
Let $G$ be a compact and connected Lie group with maximal torus $T$ and
Weyl group $W$. Then there is a homeomorphism 
$$
\Comm(G)_1/G \approx J(T)/W
$$
and a  homotopy equivalence
$$\Sigma ( \Comm(G)_1/G) \simeq \Sigma \bigvee_{n \geqslant  1} \left( \widehat{T}^n/W \right).$$
\end{thm}

Using this theorem and a few other observations, we describe 
the Poincar\'e series for representation spaces and their
infinite dimensional analogues.

\begin{thm}\label{thm: Poincare series of X(q,G) INTRO}
Let $G$ be a compact and connected Lie group with maximal torus $T$ and
Weyl group $W$.Then the Poincar\'e series of
$\Comm(G)_1/G$ is given by
\begin{align*}
\ds P(\Comm(G)_1/G;s)&=\frac{1}{|W|} \sum_{w\in W} 
					\left(\sum_{k\geqslant  0} (\det(1+sw)-1)^k\right).
\end{align*}
Moreover, we have $P(\X(q,G)_1/G;s)=P(\Comm(G)_1/G;s)$
for all $q\geqslant 2$.
\end{thm}

From this theorem and the stable decompositions in 
Theorem~\ref{thm: decomposition of X(q,G) INTRO } of the spaces $\X(q,G)_1/G$
and their corresponding Poincar\'e series in 
Theorem~\ref{thm: Poincare series of X(q,G) INTRO}, we finally
show the main theorem.

\begin{thm}\label{thm: Poincare series of Rep(Z^n,G) INTRO }
Let $G$ be a compact and connected Lie group with maximal torus
$T$ and Weyl group $W$. Then the Poincar\'e series of
 $\Rep(\Z^n,G)_1$ is given by
\begin{align*}
    P(\Rep(\Z^n,G)_1;s)&=\frac{1}{|W|} \sum_{w\in W} \det(1+sw)^n.
\end{align*}
Moreover, we have 
$P(\Rep(F_n/\Gamma^q,G)_1;s)=P(\Rep(\Z^n,G)_1;s)$
for all $q\geqslant 2$.
\end{thm}

Bergeron and Silberman \cite{bergeron2016note} show that if
$\Gamma$ is a finitely generated nilpotent group, then
$\Rep(\Gamma,G)_1$ has the same homology as $\Rep(\Gamma/\Gamma^q,G)_1$
with coefficients in a field with characteristic not dividing $|W|$.
Therefore, Theorem \ref{thm: Poincare series of Rep(Z^n,G) INTRO }
applies also to all such $\Gamma$, in particular to free
nilpotent groups $F_n/\Gamma^q$.

\subsection*{Structure of the paper}

In Section \ref{sec: Definitions} we first define and study 
topological properties of a family of spaces $\B_n(q,G,K) \subseteq G^n$
and their infinite dimensional analogue $\X(q,G,K)\subset J(G)$. 
In Section \ref{sec: G-equivar. decomp.} we show that
after one suspension we obtain a $G$-equivariant stable decomposition for each
of these spaces, proving also Theorem~\ref{thm: decomposition of X(q,G) INTRO }. 
In section \ref{sec: stable decomp of Rep and Comm mod G} we describe the stable
decomposition and the homeomorphism in Theorem~\ref{thm: Comm(G)/G = J(T)/W INTRO}.
We describe the Hilbert-Poincar\'e series for these representation 
spaces and prove Theorem \ref{thm: Poincare series of X(q,G) INTRO}
and Theorem \ref{thm: Poincare series of Rep(Z^n,G) INTRO } in
Section \ref{sec: Poincare series of Rep(Z^n,G)}. 
The last section is devoted to examples.

\subsection*{Acknowledgments} The author would like to thank
D. Ramras for the continuous and productive conversations on the subject 
and for carefully reading this manuscript.

\section{Character varieties and representation spaces}\label{sec: Rep spaces and Char var}

In this section we mention some standard notions about 
character varieties and representation spaces.
Let $G$ be a compact Lie group. By Peter-Weyl
theorem there is a faithful embedding $G \hookrightarrow \GL(n,\R)$ 
for sufficiently large $n$. This gives $G$ the structure of a
linear algebraic group. The complexification of $G$ is the group
$G'=G_\C$ given by the zero locus in $\GL(n,\C)$ 
of the defining ideal of $G$. The group $G'$ is called
a complex algebraic group and is independent of the embedding
$G \hookrightarrow \GL(n,\R)$ given by Peter-Weyl.
For example the compact Lie group $\SU(n)$ can be seen as a subgroup
of $\GL(2n,\R)$ with complexification $\SL(n,\C) \subset \GL(2n,\C).$
It is a well-known fact that the inclusion $G\hookrightarrow G'$
of $G$ to its complexification $G'$ is a homotopy equivalence.
Another fact that we should mention is the relation between 
complex reductive algebraic groups and compact Lie groups.
Namely, a complex linear algebraic group is reductive if and only 
if it is the complexification of a compact Lie group. 
For details see for instance Onischick and Vinberg~\cite{onishchik2012lie}.

Let $V$ be a complex vector space, $G\subset \GL(V)$ a complex 
reductive algebraic group, and $X \subset V$ a $G$-variety upon
which $G$ acts algebraically. This endows the ring of regular
functions $\C[X]$ with an action of $G$, with corresponding
ring of $G$-invariant regular functions denoted $\C[X]^G.$
The affine {\it geometric invariant theory quotient}, 
or GIT quotient of $X$ is defined by
$$ 
X/\!\!/G := {\rm Spec} (\C[X]^G)
$$
consisting of the closed $G$-orbits.

If $G$ is a complex reductive linear algebraic group and $\Gamma$ 
is finitely generated, then the space of homomorphisms $\Hom(\Gamma,G)$ 
has the structure of an affine algebraic variety endowed with an 
action of $G$ by conjugation. If $\Gamma$ has $n$-generators, then
$\Hom(\Gamma,G)$ can be seen as a subspace of $G^n$, by identifying
each homomorphism $f:\Gamma \to G$ with the $n$-tuple given by
the image of the generators of $\Gamma$ under $f$. 
The space $\Hom(\Gamma,G)\subset G^n$ can then be given the subspace topology.
Moreover, the action of $G$ on $\Hom(\Gamma,G)$ corresponds to the 
restriction of the diagonal conjugation action of $G$ on the product $G^n$.
The GIT quotient 
$$
\mathfrak{X}_G(\Gamma):= \Hom(\Gamma,G)/\!\!/G 
$$
is called the \textit{$G$-character variety of $\Gamma$} and the topological 
orbit space 
$$
\Rep(\Gamma,G):= \Hom(\Gamma,G)/G 
$$
is called the \textit{representation space of $\Gamma$ into $G$}.
If $G$ is compact then $\mathfrak{X}_G(\Gamma)=\Rep(\Gamma,G)$.
Moreover, it was shown by Florentino and Lawton \cite{FlorentinoLawton2014}
that if $K\leqslant G$ is a maximal compact subgroup, then 
$\Hom(\Z^n,G)/\!\!/G$ deformation retracts onto $\Hom(\Z^n,K)/K$.
In this article we always assume that $G$ is compact and connected,
but this discussion shows that results about representation spaces will imply
the same results about the corresponding character varieties.

\section{Spaces of homomorphisms and James construction}\label{sec: Definitions}

For a group $Q$ define the subgroups 
$\Gamma^k(Q) \leqs Q$ inductively by setting $\Gamma^1(Q)=Q$, and for all $i>0$ define
$\Gamma^{i+1}(Q):=[\Gamma^i(Q),Q].$ This way we obtain a normal series
$$
\cdots \unlhd \Gamma^{i+1}(Q) \unlhd \Gamma_p^i(Q) \unlhd 
		\cdots \unlhd \Gamma^2(Q) \unlhd  \Gamma^1(Q) = Q 
$$
called the \textit{descending central series of $Q$}. 
If $Q$ is the the free group $F_n$ we write $\Gamma^{q}=\Gamma^{q}(F_n)$. 
Using this normal series we define filtrations of $G^n$ by spaces 
of homomorphisms, and related filtrations of the infinite 
dimensional analogue space $J(G)$, which can be seen as a 
free monoid with prescribed topology, also
known as the James reduced product on $G$.

\subsection{Subspaces of $G^n$}

Suppose $K\leqs G$ is a closed central subgroup of the compact and connected
Lie group $G$, that is $K\leqs Z(G).$
Borel--Friedman--Morgan \cite{borel2002almost} and Adem--Cohen--Gomez 
\cite{adem.cohen.gomez} studied the spaces of
almost commuting tuples in $G$ to study the spaces of commuting tuples.
They are defined as follows. A $K$-almost commuting $n$-tuple in $G$ is an
$n$-tuple $(g_1,\dots,g_n)\in G^n$ such that $[g_i,g_j]\in K.$ Therefore,
the space of \textit{$K$-almost commuting $n$-tuples in $G$} is given by
$$
\B_n(G,K):=\{(g_1,\dots,g_n)\in G^n: [g_i,g_j]\in K \text{ for all }i,j\}\subset G^n
$$
with the subspace topology of $G^n$. This space can also be seen as the space
of homomorphisms $f:F_n \to G$ such that $f(\Gamma^2)\subset K$.
This point of view will allow us to define other variations of this space.
In particular, if $K=1_G$ then $\B_n(G,K)=\Hom(\Z^n,G)$. Moreover, as noted for instance in
\cite{adem.cohen.gomez}, the quotient of Lie groups $G \to G/K$ gives a principal $K^n$-bundle
$$
K^n \to G^n \to (G/K)^n
$$
which restricts to $K$-almost commuting tuples to give a principal $K^n$ bundle
$$
K^n \to \B_n(G,K)\to \Hom(\Z^n,G/K),
$$
since a homomorphism satisfying $f(\Gamma^2)\subset \{1_G\}$ is equivalent 
to a homomorphism $f:\Z^n \to G/K.$ 
Similarly, define the spaces $\B_n(q,G,K)\subset G^n$ by
$$\B_n(q,G,K):=\{f: F_n \to G \mid f(\Gamma^q)\subset K\} \subset G^n,$$
where $q=2$ gives $\B_n(2,G,K)=\B_n(G,K) \subset G^n.$
For fixed $n$ and varying values of $q$, we get a filtration of $G^n$
$$
\B_n(2,G,K) \subset \B_n(3,G,K) \subset \B_n(4,G,K) \subset \cdots \subset G^n,
$$
and for $K=1_G$ we obtain the ordinary filtration of $G^n$
$$
\Hom(F_n/\Gamma^2,G) \subset \Hom(F_n/\Gamma^3,G) 
	\subset \Hom(F_n/\Gamma^4,G) \subset \cdots \subset G^n.
$$
A special version of the following lemma was shown in \cite[Lemma 2.3]{adem.cohen.gomez}.

\begin{lem}
Let $G$ be a compact and connected Lie group and $K\leqs Z(G)$ 
a closed subgroup. Then for any $q \geqslant 2$ the bundle $G^n\to (G/K)^n$ 
induces a principal $K^n$-bundle
$$
\varphi_n\colon \B_n(q,G,K) \to \Hom(F_n/\Gamma^q,G/K).
$$
\end{lem}
\begin{proof}
The case when $q=2$ was proved in \cite[Lemma 2.3]{adem.cohen.gomez}. 
The proof for other values of $q$ is the same by noting that a 
homomorphism $f \colon F_n \to G$ satisfying $f(\Gamma^q)=\{1_G\}$ is equivalent 
to a homomorphism $f:F_n/\Gamma^q \to G/K.$
\end{proof}

The spaces $\B_\ast(q,G,K)$ have the structure of simplicial spaces
as follows. The Cartesian products $B_n G:=G^n$, for varying $n$, have the 
structure of simplicial spaces. That is, there are standard face and degeneracy
maps 
\begin{align*}
\sigma_i(g_1,\dots,g_n)&=(g_1,\dots,g_i,1_G,g_{i+1},\dots,g_n), \text{ and}\\
\delta_j(g_1,\dots,g_n)&= \begin{cases}
    (g_2,\dots,g_n) 					& \text{if } j=0,\\
    (g_1,\dots,g_j g_{j+1},\dots,g_n) 	& \text{if } 0<j<n,\\
    (g_1,\dots,g_{n-1})					& \text{if } j=n,
  \end{cases}
 \end{align*}
that satisfy the simplicial relations. Recall that
the geometric realization of $B_* G$ is the classifying space of $G$. 
The same face and degeneracy maps together with the relations restricted to 
$\B_\ast(q,G,K)$ make the spaces $\B_\ast(q,G,K)$ into $G$-simplicial spaces since
\begin{itemize}
\item the maps $\sigma_i$ and $\delta_j$ preserve the spaces $\B_\ast(q,G,K)$
as they are induced by homomorphisms between free groups,
\item these group homomorphisms satisfy the property
		$\sigma_i(\Gamma^q(F_n))\subset \Gamma^q(F_{n+1})$ and 
		$\delta_i(\Gamma^q(F_n))\subset \Gamma^q(F_{n-1})$, and
\item conjugation by $G$ leaves the spaces $\B_n(q,G,K)$ invariant.
\end{itemize}

\subsection{Subspaces of $J(G)$}

The {\it James reduced product} $J(X)$ is a classical 
construction introduced by James \cite{james1955reduced} 
and can be defined for any CW-complex $X$ with a good basepoint $*$.
In our paper $X$ is mainly a compact Lie group with basepoint
the identity element $1_G$. In general, $J(X)$ is defined as the topological space
\begin{equation}
J(X):= \bigg( \bigsqcup_{n\geqslant  0} X^n \bigg)/\sim
\end{equation}
where $\sim$ is the relation
$(\dots,*,\dots) \sim (\dots,\widehat{*},\dots)$
omitting the coordinates equal to the basepoint.
Alternatively this has the structure of the free monoid generated by the
elements of $X$ with the basepoint acting as the identity element.
The space $J(X)$ is weakly homotopy equivalent
to $\Omega\Sigma X$, the loops on the suspension of $X$ 
and the suspension of $J(X)$ is given by
\begin{equation}
\Sigma J(X) \simeq \Sigma \bigvee_{k \geqslant  1} \widehat{X}^n,
\end{equation}
where $\widehat{X}^n$ is the $n$-fold smash product of $X$.
This fact will be important below in our computations of the
Poincar\'e series. Let $R$ be a commutative ring with 1.
Bott and Samelson \cite{bott1953pontryagin} showed
that if the homology of $X$ is a free $R$-module, then
the homology of $J(X)$ with $R$ coefficients is isomorphic as an algebra
to the tensor algebra $\T[\widetilde{H}_\ast(X;R)]$ generated
by the reduced homology of $X$, another crucial fact that will be used below.

For compact and connected $G$ with maximal torus $T$ 
we want to study the spaces $J(T) \subset J(G)$ and the 
subspaces of $J(G)$ defined next, by assembling all the spaces
$\B_n(q,G,K)$ into a single infinite dimensional space.
Define 
\begin{equation}
\X(q,G,K):= \left(\bigsqcup_{n\geqslant  1} \B_n(q,G,K) \right)/\sim,
\end{equation}
where $\sim$ is the relation in the definition of the James reduced product.
For fixed $K$ we obtain a filtration 
\begin{equation}\label{eq: filtration of J(G)}
J(T) \subset \X(2,G,K) \subset \X(3,G,K) \subset \X(4,G,K)
	\subset \cdots \subset J(G).
\end{equation}
When $K=1_G$ we obtain the corresponding filtration studied in \cite{stafa.comm,ramrasstafa}, namely
\begin{equation}
J(T) \subset \X(2,G)\subset \X(3,G) \subset
\X(4,G) \subset \cdots \subset \X(\infty,G)=J(G).
\end{equation} It is important to single out the space 
$\Comm(G)=\X(2,G)$, which will play a more significant role here.
Each of the spaces $\X(q,G,K)$ plays an important role in the study
of the homomorphism spaces $\B_n(q,G,K)$, since stably all these 
homomorphism spaces can be extracted from the stable structure 
of the corresponding space $\X(q,G,K)$, a fact proved in the next section.

\begin{rmk}
The spaces $\B_n(q,G,K)$ and $\X(q,G,K)$ are not path connected in general.
For example Sjerve and Torres-Giese \cite{SjerveGiese} showed that $\Hom(\Z^n,\SO(3))$
has several path components, including the path component of the trivial
representation denoted by $\Hom(\Z^n,\SO(3))_1$, and a finite number of other
path component all homeomorphic to $S^3/Q_8$, the three sphere modulo quaternions.
It was then shown in \cite{stafa.thesis} that $\X(2,\SO(3))$
also consists of the path component of the trivial representation, denoted
$\X(2,\SO(3))_1$, and an infinite number of copies of $S^3/Q_8$. 
On the other hand, if $G=\SU(m)$, then both spaces $\Hom(\Z^n,\SU(m))$
and $\X(2,\SU(m))$ are path connected. We define the space 
$$\B_n(q,G,K)_1 \subseteq \B_n(q,G,K)$$
to be the path component of the trivial representation, which corresponds to the
$n$-tuple $(1_G,\dots,1_G)$. Similarly we define the space
\begin{equation}
\X(q,G,K)_1= \left(\bigsqcup_{n\geqslant  1} \B_n(q,G,K)_1 \right)/\sim \,\,\, \subseteq \X(q,G,K)
\end{equation}
to be the path component of the class of the trivial representation. 
\end{rmk}

\begin{rmk}
It is important to distinguish the cases when $K=1_G$, since then
$\B_n(q,G,1)=\Hom(F_n/\Gamma^q,G)$ and $\X(q,G,1)=\X(q,G)$
are the only spaces of interest, for which we will give
explicit answers for the Poincar\'e series in 
Section~\ref{sec: Poincare series of Rep(Z^n,G)}.

There is also an analogous {\it $p$-descending central series of $Q$} as follows.
Suppose $Q$ is a group and $p$ is a prime number. Define the subgroups 
$\Gamma^k_p(Q) \leqs Q$ inductively by setting $\Gamma^1_p(Q)=Q$, and 
for all $i>0$ define $\Gamma^{i+1}_p(Q):=[\Gamma^i(Q),Q](\Gamma^i(Q))^p.$
We obtain a normal series for $Q$
$$
\cdots \unlhd \Gamma^{i+1}_p(Q) \unlhd \Gamma^i_p(Q) \unlhd 
		\cdots \unlhd \Gamma^2_p(Q) \unlhd  \Gamma^1_p(Q) = Q.
$$
By abuse of terminology one can think of the ordinary 
{descending central series of $Q$} as the case $p=0$. If $Q$ is the
free group $F_n$ and $\Gamma^i_p=\Gamma^i_p(F_n)$, for fixed $p$
define the spaces
$\Hom(F_n/\Gamma^q_p,G)\subset G^n$, which for varying $n$ have the structure of
simplicial spaces with geometric realization $B(q,G,p)\subset BG.$
These spaces have very interesting features. For instance, 
if $G$ is finite the spaces $B(2,G,p)\subset BG$  is a space
assembled from the $p$-elementary abelian subgroups in $G$ and
the inclusion $B(2,G,p)\hookrightarrow BG$ can detect mod $p$ cohomology
of $G$, see \cite[Proposition 3.4]{AdemCohenTorresGiese}. It would be compelling
to understand the constructions in this paper using the $p$-descending
central series.

\end{rmk}

\section{$G$-equivariant stable decompositions}\label{sec: G-equivar. decomp.}

There is a filtration associated to the spaces $\B_n(q,G,K)$ defined above given by
$$
F_n^n(\B_n(q,G,K))\subset F_n^{n-1}(\B_n(q,G,K)) \subset \cdots 
		\subset F_n^1(\B_n(q,G,K)) \subset \B_n(q,G,K), 
$$
where $F_n^i(\B_n(q,G,K)) \subset \B_n(q,G,K)$ is the subspace consisting 
of $n$-tuples with at least $i$ coordinates the identity element $1_G$.
The pairs $(F_n^{k},F_n^{k+1})$ are NDR-pairs, by work of
Adem--Cohen ~\cite{adem2007commuting}, 
Adem--Cohen--Gomez~\cite{adem.cohen.gomez} and 
Villarreal~\cite{villarreal2016cosimplicial}. To be more precise,
the filtrations $F_n^{k+1}$ have the homotopy type of a $G$-CW complex
and it can be seen as a $G$-subcomplex of $F_n^{k}$.
Moreover, the simplicial spaces
$\B_\ast(q,G,K)$ are simplicially $G$-NDR as explained in
\cite[\S 4]{adem.cohen.gomez} and filtration quotients are given by
$$
F_n^{k}(\B_n(q,G,K))/F_n^{k+1}(\B_n(q,G,K))
	\approx \bigvee_{n \choose k} \B_{n-k}(q,G,K)/(F_{n-k}^{1}(\B_k(q,G,K)).
$$
For details see \cite[\S]{adem2007commuting}
or in \cite[\S 5]{adem.cohen.gomez}.
Therefore, we obtain the following theorem.
\begin{thm}\label{thm: stable decomposition 1}
Let $G$ be a compact and connected Lie group and $K$ a connected central subgroup.
Then there is a $G$-equivariant homotopy equivalence
$$\Sigma ( \B_n(q,G,K)) \simeq \Sigma 
		\bigvee_{1 \leqslant  k \leqslant  n} \bigvee_{n \choose k}  
				\B_k(q,G,K)/F_k^1(\B_k(q,G,K)), $$
and hence a homotopy equivalence
$$\Sigma ( \B_n(q,G,K)/G) \simeq \Sigma 
		\bigvee_{1 \leqslant  k \leqslant  n} \bigvee_{n \choose k} 
				[\B_k(q,G,K)/G]/[F_k^1(\B_k(q,G,K))/G]. $$
\end{thm}
\begin{proof}
The case when $q=2$ follows from \cite[Theorem 1.1]{adem.cohen.gomez}.
The same arguments show that this works for $q>2$, or alternatively by 
\cite[Theorem 1.3]{villarreal2016cosimplicial}
and \cite[Theorem 1.6]{adem.cohen.gitler.bahri.bendersky}
\end{proof}

Note that the first result of this nature was proved by
Adem and Cohen \cite{adem2007commuting}, where they write
the decomposition of the spaces $\Hom(\Z^n,G)$ into wedge
sums after one suspension for any closed subgroup of 
$GL(n,\C)$.

Similarly, the spaces $\X(q,G,K)$ can also be filtered by setting
$$
F_1(\X(q,G,K)) \subset F_2(\X(q,G,K)) \subset
			\cdots \subset F_\infty(\X(q,G,K))=\X(q,G,K),
$$
where for each $k$, we define $F_k(\X(q,G,K))$ as the image of
$$
\X_k(q,G,K) = \left(\bigsqcup_{k \geqslant  n\geqslant  1} \B_n(q,G,K)_1 \right)/\sim
$$
in $\X(q,G,K)$. 
The pairs $(F_{n-k}(\X(q,G,K)),F_{n-k+1}(\X(q,G,K)))$ are also
$G$-NDR-pairs, and the filtration quotients are 
clearly homeomorphic to
$$(F_{n-k}(\X(q,G,K)),F_{n-k+1}(\X(q,G,K)))
		\approx \B_{k}(q,G,K)/(F_{k}^{1}(\B_k(q,G,K)).
$$

Therefore, we obtain the following theorem.

\begin{thm}\label{thm: decomp filtration of J(G)}
Let $G$ be a compact and connected Lie group and $K$ a connected 
central subgroup. Then there is a $G$-equivariant homotopy equivalence
$$\Sigma ( \X(q,G,K)) \simeq \Sigma 
		\bigvee_{k \geqslant  1} 
			\B_k(q,G,K)/F_k^1(\B_k(q,G,K)),$$
and hence a homotopy equivalence
$$\Sigma ( \X(q,G,K)/G) \simeq \Sigma 
		\bigvee_{k \geqslant  1}
		[\B_k(q,G,K)/G]/[F_k^1(\B_k(q,G,K))/G]. $$
\end{thm}
\begin{proof}
The pairs $(F_{n-k}(\X(q,G,K)),F_{n-k+1}(\X(q,G,K)))$ are 
$G$-NDR-pairs with filtration quotient given by
$\B_k(q,G,K)/F_k^{1}$ as above. The James reduced product
$J(G)$ splits as a wedge sum after one suspension
$$
\Sigma J(G) \simeq \bigvee_{k\geqslant  1} \widehat{G}^k
$$
and the splitting for the suspension of $\X(q,G,K)\subseteq J(G)$ 
then follows by inspection.
\end{proof}

In particular, we obtain the following useful theorem.

\begin{thm}\label{thm: stable decomposition 2}
Let $G$ be a compact and connected Lie group. 
Then there is a homotopy equivalence
$$\Sigma ( \X(q,G)/G) \simeq \Sigma 
		\bigvee_{n \geqslant  1}
		\widehat{\Hom}(F_n/\Gamma^q,G)/G, $$
and in particular
$$\Sigma ( \Comm(G)/G) \simeq \Sigma 
		\bigvee_{n \geqslant  1}
				\widehat{\Hom}(\Z^n,G)/G. $$
\end{thm}

Note that the first decomposition holds also for the 
component of the trivial representation, as observed
also in \cite{ramrasstafa}. The conjugation action of $G$ 
preserves the component $\X(q,G)_1$, giving $G$-equivariant
decompositions of the suspension of $\X(q,G)_1$. 
In particular, we are interested in the following special case.

\begin{thm}\label{thm: decomp X(q,G)/G}
Let $G$ be a compact and connected Lie group.
Then there is a homotopy equivalence
$$\Sigma ( \X(q,G)_1/G) \simeq \Sigma 
		\bigvee_{n \geqslant  1} \widehat{\Hom}(F_n/\Gamma^q,G)_1/G. $$
\end{thm}

\section{The spaces $\Comm(G)_1/G$ and $\Rep(\Z^n,G)_1$}\label{sec: stable decomp of Rep and Comm mod G}

We start with the following theorem, which was mentioned in 
\cite[Remark 4, p. 746]{bairdcohomology} and in 
\cite[Theorem 1.3]{adem.cohen.gomez}.
We give a sketch of the proof here.

\begin{thm}\label{thm: decomposition rep}
Let $G$ be a compact and connected Lie group with maximal torus $T$ and
Weyl group $W$.  Then there is a homeomorphism 
$$\Rep(\Z^n,G)_1 \approx T^n/W $$ 
and a homotopy equivalence
$$\Sigma ( \Rep(\Z^n,G)_1) \simeq \Sigma 
		\bigvee_{1 \leqslant  k \leqslant  n} \bigvee_{n \choose k} 
				\widehat{T}^k/W, $$
where $W$ acts diagonally on $T^n$ and $\widehat{T}^k$.
\end{thm}
\begin{proof}
There is a map 
\begin{align*}
G\times T^n & \to \Hom(\Z^n,G)_1 \\
(g,t_1,\dots,t_n) &\mapsto (gt_1 g^{-1},\dots,gt_n g^{-1}),
\end{align*}
which is a surjection by Baird \cite{bairdcohomology}. This map factors through
the quotient by the normalizer of $T$
$$
\phi_n \colon G \times_{NT} T^n \to \Hom(\Z^n,G)_1,
$$
where $NT$ acts by left multiplication on $G$ and diagonally by conjugation on $T^n$.
This map is $G$-equivariant, where $G$ acts by left multiplication on $G$, trivially on
$T^n$ and by conjugation on $\Hom(\Z^n,G)_1$ \cite{adem.cohen.gomez}. Therefore, we 
get a commutative diagram
\begin{center}
\begin{tikzcd}
G \times_{NT} T^n \arrow{dd} \arrow{dr} \arrow{rr} 
		&	& \Hom(\Z^n,G)_1 \arrow{dd}	\\
		& T^n/W \arrow[dr,dotted] 	& 	\\
(G \times_{NT} T^n)/G \arrow{rr} \arrow[ur,dotted] & & \Rep(\Z^n,G)_1.
\end{tikzcd}
\end{center}
Injectivity of the map $T^n/W \to \Rep(\Z^n,G)_1$ follows by 
Sikora~\cite{sikora2014character}. This gives the  homeomorphism 
$T^n/W \approx \Rep(\Z^n,G)_1$. Moreover, we have 
$$\widehat{T}^n/W \approx \Rep(\Z^n,G)_1/(F_n^1(\Hom(\Z^n,G)_1)/G).$$
\end{proof}

The following theorem is the analogue of the above theorem, 
for the subspace $\Comm(G)\subseteq J(G)$.

\begin{thm}\label{thm: Comm(G)/G = J(T)/W }
Let $G$ be a compact and connected Lie group with maximal 
torus $T$ and Weyl group $W$. Then there is a homeomorphism 
$$
\Comm(G)_1/G \approx J(T)/W
$$
and a  homotopy equivalence
$$\Sigma ( \Comm(G)_1/G) \simeq \bigvee_{n \geqslant  1} \widehat{T}^{n}/W.$$
\end{thm}
\begin{proof}
The maps 
$$
\phi_n \colon G \times_{NT} T^n \to \Hom(\Z^n,G)_1,
$$
in the proof of Theorem \ref{thm: decomposition rep} combine to give a map
$$
\Phi   \colon G \times_{NT} J(T) \to \Comm(G)_1.
$$
The map $\Phi$ is $G$-equivariant, where $G$ acts by left multiplication on 
the left on the first factor and by conjugation on $\Comm(G)_1.$
Therefore, we get a homeomorphism 
$$
\Comm(G)_1/G \approx J(T)/W.
$$
The decomposition follows easily from the following 
$$\Sigma ( \Comm(G)_1/G)\simeq \Sigma (J(T)/W)
		\simeq \Sigma \bigvee_{n \geqslant  1} \widehat{T}^{n}/W .
 $$
\end{proof}

\section{{Poincar\'e series of $\Comm(G)_1/G$ and $\Rep(\Z^n,G)_1$}}
\label{sec: Poincare series of Rep(Z^n,G)}

The \textit{Poincar\'e series} of a space $X$ is the series
$$
P(X;t):=\sum_{k\geqslant  0} \rk_\Q (H_k(X;\Q)) t^k.
$$
In this section we describe the Poincar\'e series of 
$\Comm(G)_1/G$ and $\Rep(\Z^n,G)_1$. Using this and results by
others we obtain Poincar\'e series for their nilpotent versions.
We refine $P(X;t)$ to have a bi-graded 
\textit{Hilbert--Poincar\'e series} as follows.

First note that by Bott--Samelson \cite{bott1953pontryagin} the 
homology of the James reduced product $J(T)$,
where $T$ is the maximal torus of $G$, is the tensor algebra $\T[\widetilde{H}_*(T)]$ 
generated  by the reduced homology of $T$. This tensor algebra has a bigrading
$$\T[\widetilde{H}_*(T)] = \bigoplus_{i,j} \T[\widetilde{H}_*(T)]_{i,j},$$
where $\T[\widetilde{H}_*(T)]_{i,j}$ is generated by  
$j$--fold tensors of total (co)homological degree $i$.
The action of the Weyl group $W$ preserves this bigrading.
Note that the cohomology of $J(T)$ is the dual $\T^*[\widetilde{H}_*(T)]$ 
of the tensor algebra, which is isomorphic to the tensor algebra since the
homology of $T$ is torsion free. Below we find the bigraded version of the
Poincar\'e series of the $W$-invariant subalgebra
$$
P(\T[\widetilde{H}_*(T)]^W;s,t)=\sum_{i,j} \rk_\Q (\T[\widetilde{H}_*(T)]_{i,j}^W) s^i t^j,
$$
which we call a \textit{Hilbert--Poincar\'e series} 
because of the cohomology and tensor gradings.
To recover the ordinary Poincar\'e series we set $t=1$, 
since the tensor degree does not contribute to the cohomological degree.
We begin with a proposition.

\begin{prop}\label{prop: bigrading for tensor algebra}
Let $G$ be a compact and connected Lie group with maximal 
torus $T$ and Weyl group $W$. Then
$$P(\Comm(G)_1/G;s) \cong P(J(T)/W;s) \cong P(\T[\widetilde{H}_*(T)]^W;s).$$
\end{prop}
\begin{proof}
The first isomorphism follows from the Theorem~\ref{thm: Comm(G)/G = J(T)/W }.
The second homomorphism follows from the fact that 
$$ H_\ast(J(T);\Q) \cong \T[\widetilde{H}_*(T)],$$
and a result of Grothendieck \cite{grothendieck1957} that if $\Gamma$ is a finite group
acting (not necessarily freely) on a $CW$ complex $X$, then 
$H^\ast(X/W;\F)\cong H^\ast(X;\F)^W,$ where $\rm char(\F)=0$ or relatively 
prime to $|\Gamma|$. 
\end{proof}

First we determine the Hilbert-Poincar\'e series $P(\Comm(G)_1/G;s,t),$
where the bigrading comes from Proposition \ref{prop: bigrading for tensor algebra}.

\begin{thm}\label{thm: Poincare series of Comm(G)/G}
Let $G$ be a compact and connected Lie group with maximal 
torus $T$ and Weyl group $W$. Then the Hilbert-Poincar\'e 
series of $\Comm(G)_1/G$ is given by
\begin{align*}
\ds P(\Comm(G)_1/G;s,t)=\frac{1}{|W|} \sum_{w\in W} 
					\left(\sum_{k\geqslant  0} (\det(1+sw)-1)^k t^k\right).
\end{align*}
\end{thm}
\begin{proof} For simplicity set $X_{i,j}=\T[\widetilde{H}_*(T)]_{i,j}$ 
so that $X= \bigoplus_{i,j} X_{i,j}$.
The action of the Weyl group $W$ preserves both tensor and
cohomological degree, so it preserves each $X_{i,j}$, hence
the direct sum decomposition $\bigoplus_{i,j} X_{i,j}$. 
We want to determine 
$$
P(\Comm(G)_1/G;s,t) = P(\T[\widetilde{H}_*(T)]^W;s,t) = P(X^W;s,t).
$$
As shown in \cite[Appendix p.406]{stafa.comm} we have 
\begin{equation}
P(X^W;s,t)=\frac{1}{|W|} \sum_{w\in W} 
					\left(\sum_{i,j=0}^\infty {\rm Trace}(w|_{X_{i,j}})s^i t^j\right).
\end{equation}
From \cite[Example 3.25, p. 49]{broue.reflextion.gps}, we have 
\begin{equation}
\sum_{i=0}^n {\rm Trace}(w|_{\wedge^i \R^n})s^i = \det(1+sw)
\end{equation}
and
\begin{equation}
\sum_{i,j\geqslant 0} {\rm Trace}(w|_{X_{i,j}})s^i t^j = \frac{1}{1-t(\det(1+sw)-1)}.
\end{equation}
Therefore,
\begin{align*}
P(X^W;s,t)	&=\frac{1}{|W|} \sum_{w\in W} \left(\frac{1}{1-t(\det(1+sw)-1)}\right)\\
			&=\frac{1}{|W|} \sum_{w\in W} \left(\sum_{k\geqslant  0}(\det(1+sw)-1)^k t^k\right).
\end{align*}
\end{proof}

Now using Theorem \ref{thm: Poincare series of Comm(G)/G} we
would like to prove Theorem~\ref{thm: Poincare series of Rep(Z^n,G) INTRO }. 
We begin by proving the following proposition.

\begin{prop}\label{prop: homology of hat T^n/W}
The Poincar\'e series of $\widehat{T}^k/W$ is given by
\begin{equation}
P(\widehat{T}^k/W;s) =
\frac{1}{|W|} \sum_{w\in W} (\det(1+sw)-1)^k.
\end{equation}
\end{prop}

\begin{proof}
We first rearrange the terms in the Hilbert--Poincar\'e series
and write
\begin{align*}
\ds P(\Comm(G)_1/G;s,t)
	 &= \frac{1}{|W|} \sum_{w\in W} 
					\left(\sum_{k\geqslant  0} (\det(1+sw)-1)^k t^k\right)\\
	 &= \sum_{k=0}^\infty \left( \frac{1}{|W|} \sum_{w\in W} (\det(1+sw)-1)^k \right) t^k.
\end{align*}
The bigrading in the (co)homology of $\Comm(G)_1/G$ comes 
from the homeomorphism $\Comm(G)_1/G \approx J(T)/W.$
There is an induced homotopy equivalence
$$\Sigma ( \Comm(G)_1/G) \simeq \Sigma J(T)/W 
		\simeq \bigvee_{k\geqslant 1}  \widehat{T}^k/W$$
in Theorem \ref{thm: Comm(G)/G = J(T)/W }, and the proposition
follows since the cohomology of $\widehat{T}^k$ is concentrated 
in tensor degree $k$, and $W$ preserves tensor degree.
\end{proof}

Now we can prove the main theorem, which is only a corollary
at this point.

\begin{thm}\label{thm: Poincare series of Rep(Z^n,G)}
Let $G$ be a compact and connected Lie group with maximal 
torus $T$ and Weyl group $W$. Then the Hilbert-Poincar\'e series of
$\Rep(\Z^n,G)_1$ is given by
\begin{align*}
    P(\Rep(\Z^n,G)_1;s,t)&=\frac{1}{|W|} \sum_{w\in W} 
					\left(\sum_{k=0}^n {{n\choose k}}(\det(1+sw)-1)^k t^k\right),
\end{align*}
and the ordinary Poincar\'e series is given by
\begin{align*}
    P(\Rep(\Z^n,G)_1;s)&=\frac{1}{|W|} \sum_{w\in W} \det(1+sw)^n.
\end{align*}
\end{thm}
\begin{proof}
From Theorem \ref{thm: decomposition rep} and
Propositon \ref{prop: homology of hat T^n/W} we have the first
part of the theorem
\begin{align*}
    P(\Rep(\Z^n,G)_1;s,t)&=\frac{1}{|W|} \sum_{w\in W} 
		\left(\sum_{k=0}^n {{n\choose k}}(\det(1+sw)-1)^k t^k\right).
\end{align*} 
For the ordinary Poincar\'e series we set $t=1$. 
The second part then follows from
$$
\sum_{k=0}^n {{n\choose k}}x^k=(x+1)^n
$$
by letting $x=\det(1+sw)-1.$
\end{proof}

\begin{rmk}
In the examples in Section \ref{sec: examples Poincare}, one can easily 
verify that for the specific Lie groups there the Euler characteristic of
$\Rep(\Z^n,G)_1$ is 0.
\end{rmk}

Now we turn to finitely generated nilpotent groups.
The following was shown by Bergeron and Silberman 
\cite[Proposition 3.1]{bergeron2016note}.

\begin{prop}
Let $G$ be a compact connected Lie group. If $Q$ is
a finitely generated nilpotent group, then, 
for all $i\geqslant  2$, the inclusion
$$\Hom(Q/\Gamma^i(Q),G) \overset{\iota}{\maps} \Hom(Q,G)$$
is a homeomorphism onto the union of those components 
of the target intersecting the image of $\iota$.
\end{prop}

In particular, since $F_n/\Gamma^2 = \Z^n$ the inclusion 
$$\Hom(\Z^n,G)_1 \overset{\iota}{\maps} \Hom(F_n/\Gamma^i,G)_1$$
is a homeomorphism, and we obtain
$$\Hom(\Z^n,G)_1/G \simeq \Hom(F_n/\Gamma^i,G)_1/G.$$

\begin{cor}
Let $\Gamma$ be a finitely generated nilpotent group, 
with $\rk H_1(\Gamma)=N$, and $G$ a
reductive algebraic group. Then the Hilbert--Poincar\'e
series of $\Rep(\Gamma,G)_1$ is given by
\begin{align*}
    P(\Rep(\Gamma,G)_1;s,t)&=\frac{1}{|W|} \sum_{w\in W} \det(1+sw)^N.
\end{align*}
In particular, this is true for
\begin{itemize}
\item the free nilpotent group $F_n/\Gamma^q$ with rank $n$,
\item the surface groups $\Sg_g$ modulo $\Gamma^q(\Sg_g)$, with rank $2g$,
\item the braid groups $B_k$ modulo $\Gamma^q(B_k)$, with rank 1.
\end{itemize}
\end{cor}

\begin{rmk}\label{rmk: rep of surface groups}
If we consider the descending central series of $\Sg_g$, then there are 
\textit{inclusion} maps
$$
\Rep(\Z^{2g},G)_1 \to \Rep(\Sg_g/\Gamma^3,G)_1 \to \Rep(\Sg_g/\Gamma^3,G)_1
		\to \cdots \to \Rep(\Sg_g,G)
$$
each of them inducing isomorphisms in cohomology with field coefficients with
characteristic 0 or relatively prime to $|W|.$ 
It is known that $\Rep(\Sg_g,G)=\Rep(\Sg_g,G)_1$ is path connected. 
It would be interesting to compare the answer to the cohomology of 
$\Rep(\Sg_g,G)$ with the same coefficients, given in a result of 
Cappell, Lee and Miller~\cite[Theorem 2.2]{CappellLeeMiller}.
\end{rmk}

Finally we can determine also the Hilbert-Poincar\'e series of $\X(q,G)_1/G.$

\begin{thm}
Let $G$ be a compact and connected Lie group with maximal torus $T$
and Weyl group $W$. For all $q\geqslant 2$ there is a homotopy equivalence
$$\Sigma \, \X(q,G)_1/G \simeq  \Sigma \, \X(q+1,G)_1/G.$$
In particular, for all $q\geqslant 2$, the Hilbert-Poincar\'e 
series of $\X(q,G)_1/G$ is given by
\begin{align*}
P(\X(q,G)_1/G;s,t)=\frac{1}{|W|} \sum_{w\in W} 
					\left(\sum_{k\geqslant  0} (\det(1+sw)-1)^k t^k\right).
\end{align*}
\end{thm}
\begin{proof}
Consider the following commutative diagram of cofibrations
\begin{center}
\begin{tikzcd}
S_{n,2}(G) \arrow{d}{i} \arrow{r}{i} & \Hom(\Z^n,G)_{1} \arrow{d}{i} \arrow{r} &\widehat{\Hom}(\Z^n,G)_{1} \arrow{d}\\
S_{n,q}(G)  \arrow{r}{i} & \Hom(F_n/\Gamma^q_n,G)_{1}\arrow{r} &
		\widehat{\Hom}(F_n/\Gamma^q_n,G)_{1},\\
\end{tikzcd}
\end{center}
considered also in \cite{ramrasstafa},
where $S_{n,q}(G)$ is the subspace of $\Hom(F_n/\Gamma^q_n,G)_{1}$
consisting of $n$-tuples with at least one coordinate the identity,
and $q \geqslant  2$. 
The first two vertical maps are inclusions and are homotopy equivalences 
by the gluing lemma and \cite{bergeron2016note}. It follows that
the third vertical map is also a homotopy equivalence. The first vertical map
and the second vertical maps are $G$-equivariant by results
of Bergeron and Silberman~\cite{bergeron2016note}.
Therefore, obtain a commutative diagram
\begin{center}
\begin{tikzcd}
S_{n,2}(G)/G \arrow{d}{i} \arrow{r}{i} & \Hom(\Z^n,G)_{1}/G \arrow{d}{i} \arrow{r} &\widehat{\Hom}(\Z^n,G)_{1}/G = \widehat{T}^n/W \arrow{d}\\
S_{n,q}(G)/G  \arrow{r}{i} & \Hom(F_n/\Gamma^q_n,G)_{1}/G \arrow{r} &
		\widehat{\Hom}(F_n/\Gamma^q_n,G)_{1}/G.\\
\end{tikzcd}
\end{center}
It follows that the right vertical map is also a homotopy equivalence 
$$\widehat{T}^n/W\simeq\widehat{\Hom}(F_n/\Gamma^q_n,G)_{1}/G.$$
The rest follows  from the decomposition in 
Proposition~\ref{prop: homology of hat T^n/W} and
Theorem~\ref{thm: decomp X(q,G)/G}.
\end{proof}

\section{Examples of Poincar\'e series for $\Rep(\Z^n,G)_1$}\label{sec: examples Poincare}

We give examples of Poincar\'e series for $G$ 
one of the Lie groups $\SU(2)$, $\U(2)$, $\U(3)$, $\U(4)$, and $G_2$.
Biswas, Lawton and Ramras~\cite{biswas2014fundamental} show that
the fundamental group of the represetnation space is given by
$$\pi_1 (\Rep(\Z^n, G)_1 \cong \pi_1(G/[G,G])^n.$$
In degree 1, the formulas below recover the rank of this group.
Also $\det(1+sw)$ is invariant under conjugation, so we only need
to know the decomposition of the Weyl group into conjugacy classes.
The bigraded Hilbert--Poincar\'e series are left to the reader as
they can be easily obtained from Theorem \ref{thm: Poincare series of Rep(Z^n,G)}.

\begin{ex}
The group $G=\SU(2)$ has maximal torus of rank 1 and Weyl group 
$\Z_2=\{1,-1\}$ as a subgroup of $GL(\mathfrak{t}^\ast)$.
The representation space $\Rep(\Z^n,\SU(n))$ is path connected.
Therefore, after setting $t=1$ in Theorem \ref{thm: Poincare series of Rep(Z^n,G)}, 
the Poincar\'e series is given by
\begin{align*}
P(\Rep(\Z^n,\SU(2));s)= \frac{1}{2}\bigg((1+s)^n+(1-s)^n \bigg).
\end{align*}
\end{ex}

\begin{ex}
The group $G=\U(2)$ has maximal torus of rank 2 and Weyl group 
$\Z_2 = \left\{1,w = \begin{psmallmatrix}0 & 1\\1 & 0\end{psmallmatrix}
\right\},$
and $\Rep(\Z^n,\U(2))$ is path connected.
Therefore, we get the following  Poincar\'e polynomial
\begin{align*}
P(\Rep(\Z^n,\U(2));s) = \frac{1}{2}\bigg((1+s)^{2n} + (1-s^2)^n \bigg).
\end{align*}
\end{ex}

\begin{ex} For $G=\U(3)$ the rank is 3 and the Weyl group is
the symmetric group on 3 letters
$$W=\Sigma_3=\{e,(12),(13),(23),(123),(132)\}.$$
The matrix representations $W \leqs \GL(\mathfrak{t}^\ast)$ can
be obtained by applying each permutation in $\Sigma_3$ to the
rows of the $3\times 3$ identity matrix. This can be
done in general for the Weyl group $\Sigma_n$ of $\U(n).$
There are three conjugacy classes of elements in $\Sigma_3$, depending
on the unordered partitions of $3$, that is the class of the trivial element,
the class of $(12)\in \Sigma_3$ containing three elements,
and the class of $(123)\in \Sigma_3$ containing two elements.
Therefore, the Poincar\'e series equals
\begin{align*}
P(\Rep&(\Z^n,\U(3));s)=\frac{1}{6}\bigg((1+s)^{3n} + 3 (1-s^2)^n(1+s)^n + 2 (1+s^3)^n \bigg).
\end{align*}
\end{ex}

\begin{ex} $G=\U(4)$ has rank 4 and Weyl group 
the symmetric group on 4 letters.
The matrix representations $W \leqs GL(\mathfrak{t}^\ast)$ can
be obtained by applying each permutation in $\Sigma_4$ to the
rows of the $4\times 4$ identity matrix. 
There are five conjugacy classes of elements in $\Sigma_4$, depending
on the unordered partitions of $4$, namely 
$[1],[(12)],[(123)],[(12)(34)],[(1234)]$, each having 1, 6, 8, 3, 6
elements, and $\det(1+sw)$ equal to 
$(1+s)^4,(1-s^2)(1+s^2),(1+s^3)(1+s),(1-s^2)^2,1-s^4$, respectively.
Therefore, the Poincar\'e series is given by
\begin{align*}
P(\Rep(\Z^n,\U(4));s)=\frac{1}{24} \bigg( &(1+s)^{4n} + 6 (1-s^2)^n(1+s)^{2n} + 8 (1+s^3)^n (1+s)^n\\
	&  + 3 (1-s^2)^{2n} + 6 (1-s^4)^n \bigg).
\end{align*}
\end{ex}

\begin{ex} Now consider the exceptional Lie group $G_2$,
which has rank 2 and Weyl group the dihedral group 
$W=D_{12}$ of order 12, with presentation
$$W=D_{12}=\langle s,t | s^2,t^6,(st)^2 \rangle=
		\{1,t,t^2,t^3,t^4,t^5,s,st,st^2,st^3,st^4,st^5\}.$$
As a subgroup of $GL(\mathfrak{t}^\ast)$ we set
$$s=\left(
\begin{array}{cc}
1 & 0 \\
0 & -1 \\
\end{array}
\right),
\text{ and }
t=\frac{1}{2}\left(
\begin{array}{cc}
1 & \sqrt{3} \\
-\sqrt{3} & 1 \\
\end{array}
\right).$$
$G_2$ has a non-toral elementary abelian 2--subgroup of rank 3, 
so $\Rep(\Z^n,G_2)$ is not path-connected .
Setting $t=1$ and after simplifying, the Poincar\'e series of $\Rep(\Z^n,G_2)_1$ is
\begin{align*}
P(\Rep(\Z^n,G_2)_1;s)= \frac{1}{12} \bigg(&(1+s)^{2n}+6(-s^2+1)^n  + (-1+s)^{2n} \\ 
						&+2(s^2+s+1)^n +2(s^2-s+1)^n  \bigg).
\end{align*}
\end{ex}

\begin{ex} The Poincar\'e series of $\Comm(\U(3))/\U(3)$ can be given by
Theorem~\ref{thm: Poincare series of X(q,G) INTRO} 
$$
P(\Comm(\U(3))/\U(3);s)= \frac{1}{6} \sum_{w\in \Sigma_3} \sum_{k\geqslant  0} (\det(1+sw)-1)^k.
$$
Therefore, using the data in the previous examples we have 
$$
P(\Comm(\U(3))/\U(3);s)= \frac{1}{6}\sum_{k\geqslant  0} 
\bigg((3s+3s^2+s^3)^{k} + 3 (s-s^2-s^3)^k + 2 s^{3k}\bigg).
$$
\end{ex}


\begin{rmk}
The same information can be used to obtain the same Poincar\'e series for 
$\Rep(F_n/\Gamma^q,G)$ and $\Rep(\Sg_g/\Gamma^q,G)$ (and other examples) as shown above.
Moreover, the same information can be used to obtain Hilbert--Poincar\'e series for
$\X(m,G)_1/G$ for all $m\geqslant  2$.

Another observation is that our calculations in this section indicate that
for even $n$, and $G$ as above, the spaces $\Rep(\Z^n,G)_1$ 
are rational Poincar\'e duality spaces, contrary to $\Hom(\Z^n,G)_1$, 
for which this happens when $n$ is odd \cite{ramrasstafa}.
\end{rmk}

\bibliographystyle{amsplain}
 
\bibliography{CharVar}

\end{document}